\documentclass[12pt]{article}
\usepackage{amsmath,amssymb,amsfonts}
\usepackage{epsfig}
\usepackage{tikz}
\usepackage{graphics}
\usepackage{bbm}
\usepackage{bm, bbm}
\usepackage{todonotes}
\usepackage[thmmarks]{ntheorem}
\usepackage[]{hyperref}
 \hypersetup{colorlinks=true, citecolor=black,%
                 filecolor=black,%
                 linkcolor=black,%
                 urlcolor=black,  pdftex}

\usetikzlibrary[graphs, arrows, backgrounds, intersections, positioning, fit, petri, calc, shapes, decorations.pathmorphing]
\usepgflibrary{patterns}

\definecolor{darkred}{RGB}{105,0,0}

\makeatletter
\newtheoremstyle{prime}%
 {\item[\hskip\labelsep \theorem@headerfont ##1\ \theorem@separator]}%
{\item[\hskip\labelsep \theorem@headerfont ##1\ ##3' \theorem@separator]}
\makeatother

\makeatletter
\newtheoremstyle{proofof}
{\item[\hskip\labelsep \theorem@headerfont ##1\ \theorem@separator]}%
{\item[\hskip\labelsep \theorem@headerfont ##1\ ##3\theorem@separator]}
\makeatother

\newtheorem{theorem}{Theorem}
\newtheorem{lemma}[theorem]{Lemma}

\newtheorem{proposition}[theorem]{Proposition}

\newtheorem{claim}{Claim}

\theoremstyle{prime}

\def \QD1 {\hfill $\spadesuit$}

\newcommand{\DF}[1]{{\bf #1\/}}

\newcommand{\set}[2]{\{#1 \;|\; #2 \}}
\newcommand{\ems}{\varnothing}
\newcommand{\sm}{\setminus}

\newcommand{\De}{\Delta}
\newcommand{\de}{\delta}
\newcommand{\om}{\omega}
\newcommand{\cn}{\chi}
\newcommand{\f}{\varphi}

\newcommand{\cC}{\mathcal{CO}}

\newcommand{\cH}{{\cal H}}
\newcommand{\cc}{{\cal C}}

\numberwithin{equation}{section}

\theoremstyle {nonumberplain}
\theoremseparator{:}
\theorembodyfont{\normalfont}
\theoremsymbol{\rule{1ex}{1ex}}
\newtheorem{proof}{Proof}

\theoremstyle{proofof}
\newtheorem{proofof}{Proof of}

\theoremsymbol{$\square$}
\newtheorem{proof2}{Proof}

\begin{document}
\title{\bf Coloring hypergraphs of low connectivity}

\author{{{
Thomas Schweser}\thanks{The authors thank the Danish Research Council for support by the program Algodisc.}
\thanks{
Technische Universit\"at Ilmenau, Inst. of Math., PF 100565, D-98684 Ilmenau, Germany. E-mail
address: thomas.schweser@tu-ilmenau.de}}
\and{{Michael Stiebitz}\footnotemark[1]
\thanks{
Technische Universit\"at Ilmenau, Inst. of Math., PF 100565, D-98684 Ilmenau, Germany. E-mail
address: michael.stiebitz@tu-ilmenau.de}}
\and{{Bjarne Toft}\footnotemark[1]
\thanks{
University of Southern Denmark, IMADA, Campusvej 55, DK-5320 Odense M, Denmark. E-mail address: btoft@imada.sdu.dk}}}

\date{}
\maketitle

\begin{abstract}
For a hypergraph $G$, let $\chi(G), \Delta(G),$ and $\lambda(G)$ denote the chromatic number, the maximum degree, and the maximum local edge connectivity of $G$, respectively. A result of Rhys Price Jones from 1975 says that every connected hypergraph $G$ satisfies $\chi(G) \leq \Delta(G) + 1$ and equality holds if and only if $G$ is a complete graph, an odd cycle, or $G$ has just one (hyper-)edge. By a result of Bjarne Toft from 1970 it follows that every hypergraph $G$ satisfies $\chi(G) \leq \lambda(G) + 1$. In this paper, we show that a hypergraph $G$ with $\lambda(G) \geq 3$ satisfies $\chi(G) = \lambda(G) + 1$ if and only if $G$ contains a block which belongs to a family $\mathcal{H}_{\lambda(G)}$. The class $\mathcal{H}_3$ is the smallest family which contains all odd wheels and is closed under taking Haj\'os joins. For $k \geq 4$, the family $\mathcal{H}_k$ is the smallest that contains all complete graphs $K_{k+1}$ and is closed under Haj\'os joins.
For the proofs of the above results we use critical hypergraphs. A hypergraph $G$ is called $(k+1)$-critical if $\cn(G)=k+1$, but $\cn(H)\leq k$ whenever $H$ is a proper subhypergraph of $G$. We give a characterization of $(k+1)$-critical hypergraphs having a separating edge set of size $k$ as well as a a characterization of $(k+1)$-critical hypergraphs having a separating vertex set of size $2$.
\end{abstract}



\section{Introduction and main results}
In the 1960s, Erd\H os and Hajnal \cite{ErdHaj66} introduced a coloring concept for hypergraphs. A \textbf{coloring} of a hypergraph $G$ with \textbf{color set} $C$ is a function $\varphi:V(G) \to C$ such that for each edge $e \in E(G)$ there are vertices $v,w \in e$ with $\varphi(v) \neq \varphi(w)$. Since each edge of a graph contains exactly two vertices, this concept is a generalization of the usual coloring concept for graphs. The \textbf{chromatic number} $\chi(G)$ of a hypergraph $G$ is the least integer $k$ such that $G$ admits a \textbf{$k$-coloring}, that is, a coloring with color set $\{1,2,\ldots,k\}$. This definition enables the transfer of various famous results on colorings of graphs to the hypergraph case.
But even if one is only interested in graphs, the study of hypergraphs may be of use and help in many cases, as demonstrated for example in \cite{Toft74}. Brooks' theorem \cite{brooks} was extended to hypergraphs by Jones \cite{Jones} in 1975.

\begin{theorem}
Let $G$ be a connected hypergraph. Then, $\chi(G) \leq \Delta(G)+1$ and equality holds if and only if $G$ is a complete graph, an odd cycle, or $G$ contains exactly one edge.
\end{theorem}

In this paper, we examine the relation between the chromatic number of a hypergraph and its edge connectivity. Let $G$ be a hypergraph with at least two vertices. The \textbf{local edge connectivity} $\lambda_G(v,w)$ of distinct vertices $v,w$ in the hypergraph $G$ is the maximum number of edge-disjoint $(v,w)$-hyperpaths of $G$. The \emph{maximum local edge connectivity} of a hypergraph $G$ is
$$\lambda(G) = \max \{\lambda_G(v,w) ~|~ v,w \in V(G), v \neq w\}.$$
If $G$ has at most one vertex, we set $\lambda(G)=0$. By a result of Toft \cite{Toft70}, each hypergraph $G$ satisfies $\chi(G) \leq \lambda(G) + 1$. Our aim is to characterize the class of hypergraphs for which equality hold. To this end, we use a famous construction by Haj\'os \cite{Hajos61}, which was extended to hypergraphs by Toft~\cite{Toft74}.

Let $G_1$ and $G_2$ be two vertex disjoint hypergraphs and, for $i\in \{1,2\}$, let $e_i \in E(G_i)$ and $v_i \in e_i$. Then, we create a new hypergraph $G$ by deleting $e_1$ and $e_2$, identifying the vertices $v_1$ and $v_2$ to a new vertex $v^*$, and adding a new edge $e^* \in E(G)$ either with $e^* = (e_1 \cup e_2) \setminus \{v_1,v_2\}$ or with $e^*=(e_1 \cup e_2 \cup v^*) \setminus \{v_1,v_2\}$. Then, $G$ is a \textbf{Haj\'os join} of $G_1$ and $G_2$ and we write $G=(G_1,v_1,e_1) \Delta (G_2,v_2,e_2)$ or, briefly, $G=G_1 \Delta G_2$. Figure~\ref{fig_haj-join} shows the two possible Haj\'os joins of two $K_4$.

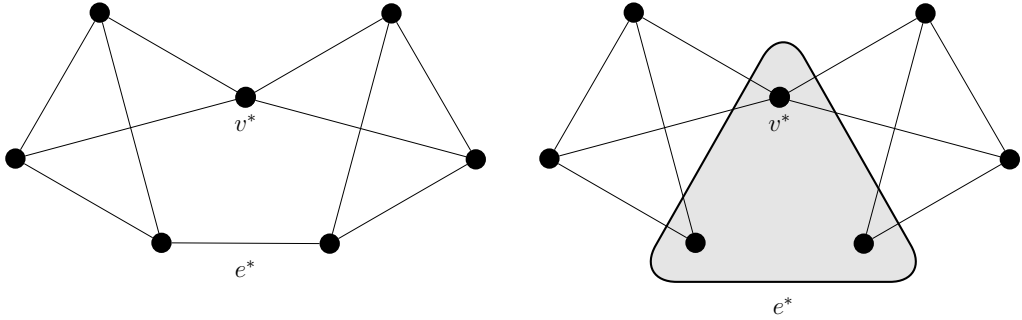
\begin{figure}
\centering
\resizebox{\linewidth}{!}{

\begin{tikzpicture} [node distance=1cm, bend angle=20,
vertex/.style={circle,minimum size=3mm,very thick, draw=black, fill=black, inner sep=0mm}, information text/.style={fill=red!10,inner sep=1ex, font=\Large}, help lines/.style={-,color=black}]

\node[draw=none,minimum size=4cm,regular polygon,regular polygon sides=4, rotate=-30] (a) {};
\foreach \x in {1,2,3,4}
\node[vertex] (a\x) at (a.corner \x) {};
\node [label={below:$v^*$}] at (a1) {};

\node[draw=none,minimum size=4cm,regular polygon,regular polygon sides=4, rotate=30, xshift=3.35cm, yshift=-1.95cm](b){};
\foreach \x in {1,2,3,4}
\node[vertex] (b\x) at (b.corner \x) {};

\path[-]

(a1)
	 edge[help lines] (a2)
	 edge[help lines] (a3)
(a2) edge[help lines] (a3)
	 edge[help lines] (a4)
(a3) edge[help lines] (a4)

(b1)
	 edge[help lines] (b2)
	 edge[help lines] (b3)
	 edge[help lines] (b4)
(b2) edge[help lines] (b4)
(b3) edge[help lines] node[midway, label={below:$e^*$}]{} (a4)
	 edge[help lines] (b4);

\begin{scope}[xshift=9cm]

\node[draw=none,minimum size=4cm,regular polygon,regular polygon sides=4, rotate=-30] (a) {};
\foreach \x in {1,2,3,4}
\node[vertex] (a\x) at (a.corner \x) {};
\node [label={below:$v^*$}] at (a1) {};

\node[draw=none,minimum size=4cm,regular polygon,regular polygon sides=4, rotate=30, xshift=3.35cm, yshift=-1.95cm](b){};
\foreach \x in {1,2,3,4}
\node[vertex] (b\x) at (b.corner \x) {};
\node at (2,-3) {$e^*$};

\path[-]

(a1)
	 edge[help lines] (a2)
	 edge[help lines] (a3)
(a2) edge[help lines] (a3)
	 edge[help lines] (a4)
(a3) edge[help lines] (a4)

(b1)
	 edge[help lines] (b2)
	 edge[help lines] (b3)
	 edge[help lines] (b4)
(b2) edge[help lines] (b4)
(b3) edge[help lines] (b4);

\begin{pgfonlayer}{background}
\filldraw[fill=black!50!, fill opacity=.2,rounded corners=20pt, line width=1pt, draw=black] (-0.5,-2.6)--(4.5,-2.6)--(2,1.8)--cycle;

\end{pgfonlayer}

\end{scope}

\end{tikzpicture}
%
}
\caption{The two possible Haj\'os joins of two $K_4$.}
\label{fig_haj-join}

\end{figure}

For an integer $k \geq 3$ we define a class $\mathcal{H}_k$ of hypergraphs as follows.
Let $\mathcal{H}_3$ be the smallest class of hypergraphs that contains all odd wheels and is closed under taking Haj\'{o}s joins. Moreover, for $k \geq 4$, let $\mathcal{H}_k$ be the smallest class of hypergraphs that contains all complete graphs of order $k+1$ and is closed under taking Haj\'{o}s joins.

Recall that a \textbf{block} of a hypergraph $G$ is a maximal connected subhypergraph of $G$ that does not contain a separating vertex. It is well known that any two blocks of $G$ have at most one vertex in common. In particular,
\begin{align}\label{eq_block-chi}
\chi(G) = \max \{\chi(B) ~ | ~ B \text{ is a block of }  G\}.
\end{align}
This is due to the fact that if we have optimal colorings of the blocks of $G$, then, by permuting the colors in the blocks, we can create an optimal coloring of $G$.

The next theorem is the main result of this paper, it is a generalization of Brooks' theorem for hypergraphs. The graph-counterpart was proven by Aboulker, Brettell, Havet, Marx, and Trotignon \cite{AlboukerV2016} for $\lambda(G)=3$ and by Stiebitz and Toft \cite{StiToft18} for $\lambda(G) \geq 4$.
\begin{theorem}\label{theorem_main-result}
Let $G$ be a hypergraph with $\lambda(G) \geq 3$. Then, $\chi(G) \leq \lambda(G) + 1$ and equality holds if and only if $G$ has a block belonging to the class $\mathcal{H}_{\lambda(G)}$.
\end{theorem}

Note that for $\lambda(G) \in \{0,1\}$, it is obvious that a connected hypergraph $G$ satisfies $\chi(G) = \lambda(G) + 1$ if and only $\lambda(G)=0$ and $G=K_1$, or $\lambda(G)=1$ and each block of $G$ consists of just one edge. The case $\lambda(G)=2$ has not yet been solved in a satisfactory way, that is, we do not know with certainty what $\cH_2$ is.

\section{Notation and Basic Concepts}
A \textbf{hypergraph} is a pair $G=(V,E)$, where $V$ and $E$ are two finite sets, $E \subseteq 2^V$, and $|e| \geq 2$ for all $e \in E$. Then, $V(G)=V$ is the \textbf{vertex set} of $G$ and its elements are the \textbf{vertices} of $G$. Furthermore, $E(G)=E$ is the \textbf{edge set} of $G$; its elements are the \textbf{edges} of $G$. The \textbf{empty} hypergraph is the hypergraph $G$ with $V(G)=E(G)=\ems$; we denote it by $G=\ems$. A \textbf{simple} hypergraph is a hypergraph in which no edge is contained in another edge. Note that hypergraphs in this paper have no multiple edges.

For a hypergraph $G$ we use the following notation. The \textbf{order} $|G|$ of $G$ is the number of vertices of $G$. Let $e$ be an arbitrary edge of $G$. If $|e| \geq 3$, the edge $e$ is said to be a \textbf{hyperedge}, otherwise, for $|e|=2$, $e$ is an \textbf{ordinary} edge.  If $e$ is an ordinary edge of $G$ with $e=\{v,w\}$, we briefly write $e=vw$ and $e=wv$. As usual, we write $G=K_n$ if $G$ is a complete graph of order $n$ and $G=C_n$ if $G$ is a cycle of order $n$ consisting only of ordinary edges. A cycle is called \textbf{odd} or \textbf{even} depending on whether its order is odd or even. An \textbf{odd wheel} is a graph obtained from an odd cycle by adding one vertex and joining it to all others. A \textbf{hyperwheel} is a hypergraph obtained from an edge by adding one vertex and joining it to all vertices of the edge by ordinary edges.


For a hypergraph $G$ and a vertex set $X \subseteq V(G)$, let
$$\partial_G(X)=\set{e\in E(G)}{e \cap X \neq \varnothing ~\text{and}~ e \cap (V(G) \setminus X) \neq \varnothing}.$$
If $X=\{v\}$ is a singleton, we just write $\partial_G(v)$.
The \textbf{degree} of $v$ in $G$ is defined as $d_G(v)=|\partial_G(v)|$.  As usual, $\de(G)=\min_{v\in V(G)} d_G(v)$ is the \DF{minimum degree} of $G$ and $\De(G)=\max_{v\in V(G)}d_G(v)$ is the \DF{maximum degree} of $G$. If $G$ is empty, we set $\de(G)=\De(G)=0$. A non-empty hypergraph $G$ is said to be $r$-\textbf{regular} or, briefly,   \textbf{regular} if each vertex in $G$ has degree $r$.

A hypergraph $G'$ is a \DF{subhypergraph} of $G$, written $G'\subseteq G$, if $V(G')\subseteq V(G)$ and $E(G') \subseteq E(G)$. Moreover, $G'$ is a \textbf{proper} subhypergraph of $G$, if $G' \subseteq G$ and $G' \neq G$. Let $G_1$ and $G_2$ be two hypergraphs. Then, $G_1 \cup G_2$ denotes the \textbf{union} of $G_1$ and $G_2$, that is, the hypergraph $G'$ with $V(G')=V(G_1) \cup V(G_2)$, and $E(G')=E(G_1)\cup E(G_2)$. Similarly, $G'=G_1 \cap G_2$ denotes the \textbf{intersection} of $G_1$ and $G_2$, where $V(G')=V(G_1) \cap V(G_2)$ and $E(G')=E(G_1)\cap E(G_2)$.

Let $G$ be a hypergraph and let $X \subseteq V(G)$ be a vertex set. We consider two new hypergraphs. First, $G[X]$ is the subhypergraph of $G$ with
$$V(G[X])=X \text{ and } E(G[X])=\set{e\in E(G)}{e\subseteq X}.$$
We say that $G[X]$ is the subhypergraph of $G$ \textbf{induced} by $X$.  More general, a hypergraph $G'$ is said to be an \textbf{induced subhypergraph} of $G$ if $V(G') \subseteq V(G)$ and $G'=G[V(G')]$. Secondly, $G(X)$ is the hypergraph with
$$V(G(X))=X \text{ and } E(G(X))=\set{e \cap X}{e \in E(G) ~\text{and}~ |e \cap X|\geq 2}.$$
We say that $G(X)$ is the hypergraph obtained by \textbf{shrinking} $G$ \textbf{to} $X$. Note that $G(X)$ does not necessarily need to be a subhypergraph of $G$. As usual, we define $G-X=G[V(G) \setminus X]$ and $G \div X = G(V(G) \setminus X)$. For the sake of readability, if $X=\{v\}$ for some vertex $v$, we will write $G - v$ and $G \div v$ instead of $G - X$ and $G \div X$. To obtain the reverse operation to $G-v$, let $G'$ be a proper induced subhypergraph of $G$ and let $v \in V(G) \setminus V(G')$. Then, $G'+v=G[V(G') \cup \{v\}]$. If $F \subseteq E(G)$ is an edge set, then let $G-F$ be the hypergraph that results from $G$ by deleting all edges from $F$. If $F=\{e\}$ is a singleton, we write $G-e$ rather than $G-F$.

Let $G$ be a non-empty hypergraph. A $(v,w)$-\textbf{hyperpath} in $G$ is a sequence $(v_1, e_1, v_2, e_2, \ldots,e_{q-1}, v_q)$ of distinct vertices
$v_1, v_2, \ldots, v_{q}$ of $G$ and distinct edges
$e_1, e_2, \ldots, e_{q-1}$ of $G$ such that $v=v_1$, $w=v_q$ and
$\{v_i,v_{i+1}\}\subseteq e_i$ for $i\in \{1, 2, \ldots, q-1\}$. If $u$ and $u'$ are vertices contained in a hyperpath $P$, we will write $uPu'$ in order to denote the $(u,u')$-subhyperpath of $P$. Two hyperpaths are \textbf{edge-disjoint} if the edges from one are all different from the edges of the other. The hypergraph $G$ is \textbf{connected} if there is a hyperpath in $G$ between any two of its vertices. A (connected) \textbf{component} of $G$ is a maximal connected subhypergraph of $G$.

A \textbf{separating vertex set} of $G$ is a set $S \subseteq V(G)$ such that $G$ is the union of two induced subhypergraphs $G_1$ and $G_2$ with $V(G_1) \cap V(G_2) = S$ and $|G_i| > |S|$ for $i \in \{1,2\}$. If $S=\{v\}$ is a singleton, we say that $v$ is a \textbf{separating vertex} of $G$. Note that $S$ is a separating vertex set if and only if $G \div S$ has more components than $G$.  Finally, a \textbf{block} of $G$ is a maximal connected subhypergraph of $G$ that has no separating vertex. Thus, every block of $G$ is a connected induced subhypergraph of $G$. It is easy to see that two blocks of $G$ have at most one vertex in common, and that a vertex $v$ is a separating vertex of $G$ if and only if it is contained in more than one block.

A \textbf{separating edge set} of $G$ is a set $F \subseteq E(G)$ such that $G-F$ has more components than $G$. If $F$ is a separating edge set and there is no proper subset of $F$ that is a separating edge set, as well, $F$ is said to be a \textbf{minimal separating edge set}. It is well known that if $F$ is a minimal separating edge set of a connected hypergraph $G$, then $F = \partial_G(X)$ for some non-empty proper subset $X$ of $V(G)$. An edge $e$ is a \textbf{bridge} of a hypergraph $G$ if $G-e$ has $|e|-1$ more components than $G$. Note that an edge $e$ is a bridge if and only if each vertex from $e$ belongs to a different component of $G-e$.

A hypergraph $G$ is $k$\textbf{-edge-connected} for an integer $k\geq 1$ if $|G| \geq 2$ and $G-F$ is connected for any set $F \subseteq E(G)$ with $|F| \leq k-1$. It is well known that Menger's Theorem also holds for hypergraphs (see \cite[Theorem~2.5.28]{Frank11} and \cite{Kir03}).

\begin{theorem}\label{theorem_menger}
If $G$ is a hypergraph and $v, w$ are distinct vertices of $G$, then $$\lambda_G(v,w)=\min \{|\partial_G(X)|~|~ v \in X \subseteq V(G)\setminus \{w\}\}.$$
\end{theorem}

\section{Connectivity of critical hypergraphs}

In order to prove Theorem~\ref{theorem_main-result}, we use the concept of \textbf{critical hypergraphs}. Critical graphs were introduced by Dirac in his Ph.D. thesis and the resulting papers \cite{Dir52} and \cite{Dir53}. His concept was extended to hypergraphs by Lov\'asz \cite{Lov68}. We say that a hypergraph $G$ is $(k+1)$\textbf{-critical} or, briefly, \textbf{critical} if $\chi(G)=k+1$, but $\chi(H) \leq k$ for any proper subhypergraph $H$ of $G$. Critical hypergraphs are a useful concept in chromatic number theory as many problems can be reduced to critical hypergraphs. In particular, each hypergraph $G$ contains a critical hypergraph $H$ with $\chi(H) = \chi(G)$. The next two propositions state some well known facts about critical hypergraphs.

\begin{proposition} \label{prop_G-e}
Let $G$ be a connected hypergraph and let $k \geq 0$ be an integer. Then, $G$ is $(k+1)$-critical if and only if $\chi(G-e) \leq k < \chi(G)$ for each edge $e \in E(G)$.
\end{proposition}

It is easy to see that $K_1$ is the only $1$-critical hypergraph and that the only $2$-critical hypergraphs are the connected hypergraphs that contain only one edge. Regarding graphs, it is also easy to obtain that the only $3$-critical graphs are the odd cycles. However, it seems unlikely that there is a good characterization of $3$-critical hypergraphs as even the decision
whether a given hypergraph $G$ satisfies $\chi(G) \leq 2$ is ${\sf NP}$-complete (see \cite{Lov73}).

\begin{proposition} \label{prop_basic-facts}
Let $G$ be a $(k+1)$-critical hypergraph for some integer $k \geq 0$. Then, the following statements hold:
\begin{itemize}
\item[\upshape (a)] $\delta(G) \geq k$, in fact each vertex $v$ is contained in $k$ edges having pairwise only $v$ in common.
\item[\upshape (b)] If $k \geq 1$, then $G$ is $k$-edge-connected. In particular, $\lambda_G(v,w) \geq k$ for distinct vertices $v,w \in V(G)$.
\item[\upshape (c)] $G$ is a block.
\item[\upshape (d)] $G$ is a simple hypergraph.
\end{itemize}
\end{proposition}

Statement (a) follows from the fact that there is a coloring of $G-v$ with color set $C=\{1,2,\ldots,k\}$. This coloring, however, cannot be extended to a $k$-coloring of $G$, and therefore for each color $\alpha\in C$ there is an edge in $\partial_G(v)$ where all vertices have color $\alpha$, except $v$. This proves (a). Statement (b) was proved by Toft in \cite{Toft74}; we also give a proof in Theorem~\ref{theorem_edge-cut}. Statement (c) is a direct consequence of \eqref{eq_block-chi}, and (d) is obvious.

Proposition~\ref{prop_basic-facts}(a) leads to a classification of the vertices of critical hypergraphs. Let $G$ be a $(k+1)$-critical hypergraph. Then, a vertex is said to be a \textbf{low vertex} of $G$ if it has degree $k$ in $G$, and a \textbf{high vertex}, otherwise. Thus each high vertex of $G$ has degree at least $k+1$ in $G$.

We say that a connected hypergraph is a \textbf{Gallai tree} if each of its blocks is a complete graph, an odd cycle, or consists of just one hyperedge. A \textbf{Gallai forest} is a hypergraph whose components are all Gallai trees.
The next lemma is from Kostochka and Stiebitz \cite{KosStieb03}; it generalizes a famous result of Gallai \cite{Gal63a} on critical graphs.


\begin{lemma}\label{lemma_G(X)-Gallai}
Let $G$ be a $(k+1)$-critical hypergraph for some integer $k \geq 2$, let $L$ be the set of low vertices of $G$, and $H=V(G) \setminus L$. Moreover, let $$F=\{e \in E(G) ~|~ |e \cap L | \geq 2~ \text{and}~ |e \cap H|\geq 1\}.$$ If $L\not= \ems$, then the following statements hold:
\begin{itemize}
\item[\upshape (a)] $G(L)$ is a Gallai forest.
\item[\upshape (b)] If $e,e'\in F$ and $e \neq e'$, then $e \cap L \neq e' \cap L$.
\item[\upshape (c)] If $e \in F$, then $e \cap L$ belongs to $E(G(L))$ and is a bridge of $G(L)$.
\item[\upshape (d)] If $H$ is empty, then $G$ is a $K_{k+1}$, or $k=2$ and $G$ is an odd cycle, or $k=1$ and $G$ is a connected hypergraph consisting of one edge. Furthermore, if $G(L)$ contains a $K_{k+1}$, then $G=K_{k+1}$.
\end{itemize}
\end{lemma}

Gallai \cite{Gal63a} furthermore characterized the critical graphs having exactly one high vertex. A similar characterisation holds for hypergraphs; however, we only need the following easy consequence of the above lemma.

\begin{lemma}\label{lemma_|G_H|=1}
Let $G$ be a $(k+1)$-critical hypergraph for some integer $k \geq 2$. If $G$ has exactly one high vertex, then either $G$ has a separating vertex set of size $2$, or $k=2$ and $G$ is a hyperwheel, or $k=3$ and $G$ is an odd wheel.
\end{lemma}

\begin{proof}
Let $v$ be the only high vertex of $G$. Then, $L=V(G) \setminus \{v\}$ is the set of low vertices of $G$ and $G(L)=G\div v$. By Lemma~\ref{lemma_G(X)-Gallai}(a), $G(L)$ is a Gallai forest. As $G$ is a block (by Proposition~\ref{prop_basic-facts}(c)), $G(L)$ is connected and therefore a Gallai tree. Let $B$ be an end-block of $G(L)$. If $B$ is not the only block of $G(L)$, then $B$ contains a separating vertex $u$ of $G(L)$ and $\{v,u\}$ is a separating vertex set of $G$, so we are done. Otherwise, $G(L)=B$ and it follows from Lemma~\ref{lemma_G(X)-Gallai}(c) that $\partial_G(v)$ contains only ordinary edges and so $G(L)=G[L]$.
Since $G(L)$ is a Gallai tree consisting  only of the block $B$, this block $B$ is regular of degree $k-1$ and $v$ joined to each vertex of $B$ by an ordinary edge. As $d_G(v) \geq k+1$, $k=2$ and $B$ consists of just one edge, or $k=3$ and $B$ is an odd cycle. Thus, $k=2$ and $G$ is a hyperwheel, or $k=3$ and $G$ is an odd wheel, as claimed.
\end{proof}

As was previously noted, a critical graph is connected and contains no separating vertex. Dirac \cite{Dir52} as well as Gallai \cite{Gal63a} characterized critical graphs having a separating vertex set of size 2. The next theorem is the hypergraph counterpart. For a hypergraph $G$, by $\mathcal{CO}_k(G)$ we denote the set of all $k$-colorings of $G$, i.e., all colorings of $G$ with color set $\{1,2,\ldots,k\}$.

\begin{theorem}\label{theorem_gal+dir}
Let $G$ be a $(k+1)$-critical hypergraph for an integer $k \geq 2$, and let $S \subseteq V(G)$ be a separating vertex set of $G$ satisfying $|S| \leq 2$. Then $S$ is an independent set of $G$ consisting of two vertices, say $v$ and $w$, and $G \div S$ has exactly two components $H_1$ and $H_2$. Moreover, if $G_i=G[V(H_i) \cup S]$ for $i\in\{1,2\}$, we can adjust the notation so that for a coloring $\varphi_1 \in \mathcal{CO}_k(G_1)$ we have $\varphi_1(v)=\varphi_1(w)$. Then, the following statements hold:
\begin{itemize}
\item[\upshape (a)] Each coloring $\varphi \in \mathcal{CO}_k(G_1)$ satisfies $\varphi(v)=\varphi(w)$ and each coloring $\varphi \in \mathcal{CO}_k(G_2)$ satisfies $\varphi(v) \neq \varphi(w)$.
\item[\upshape (b)] The hypergraph $G_1'=G_1 + vw$ obtained from $G$ by adding the edge $vw$ is $(k+1)$-critical.
\item[\upshape (c)] The hypergraph $G_2'$ obtained from $G_2$ by identifying $v$ and $w$ is $(k+1)$-critical.
\end{itemize}
\end{theorem}
\begin{proof}
Since $G$ is $(k+1)$-critical with $k \geq 2$, the separating set $S$ consists of exactly two elements, say $S=\{v,w\}$. Then, $G$ is the union of two induced subhypergraphs $G_1$ and $G_2$ with $V(G_1) \cap V(G_2) = \{v,w\}$ and $|G_i| > 2$ for $i \in \{1,2\}$. Since $G_i$ is a proper subhypergraph of $G$, there is a coloring $\varphi_i \in \mathcal{CO}_k(G_i)$ ($i \in \{1,2\}$). Then, for one coloring, say $\varphi_1$, we have $\varphi_1(v)=\varphi_1(w)$ and for $\varphi_2$, we have $\varphi_2(v) \neq \varphi_2(w)$. For otherwise, we could permute the colors in one coloring such that $\varphi_1(v)=\varphi_2(v)$ and $\varphi_1(w)=\varphi_2(w)$ so that $\varphi_1 \cup \varphi_2$ would be a $k$-coloring of $G$, which is impossible. Consequently, $S$ is an independent set of $G$. Furthermore it follows that each coloring $\varphi \in \mathcal{CO}_k(G_1)$ satisfies $\varphi(v)=\varphi(w)$ and each coloring $\varphi \in \mathcal{CO}_k(G_2)$ satisfies $\varphi(v) \neq \varphi(w)$. Hence, (a) is proven.

For the proof of (b), let $G_1'=G_1 + vw$. Then, it follows from (a) that $\chi(G_1') \geq k+1$. Let $e$ be an arbitrary edge of $G_1'$. We show that $G_1' - e$ admits a $k$-coloring. If $e=vw$, this is evident. Otherwise, $e \in E(G_1)$ and there is a $k$-coloring $\varphi$ of $G-e$. By (a), it follows that $\varphi(v) \neq \varphi(w)$ and so $\varphi$ induces a $k$-coloring of $G_1'-e$. Hence, $G_1'$ is $(k+1)$-critical (see Proposition~\ref{prop_G-e}).

In order to prove (c), let $G_2'$ be the hypergraph obtained from $G_2$ by identifying $v$ and $w$ to a new vertex $v^*$. Then, by (a), $\chi(G_2') \geq k+1$. Let $e$ be an arbitrary edge of $G_2'$ and let $e'$ be a corresponding edge of $G_2$. Then, $G-e'$ admits a $k$-coloring $\varphi$ and, by (a), $\varphi(v)=\varphi(w)$ and so $\varphi$ induces a $k$-coloring of $G_2'-e$. Hence, $G_2'$ is $(k+1)$-critical.

Finally, we obtain that $$G \div S = (G_1 \div S) \cup (G_2 \div S) = (G_1' \div S) \cup (G_2' \div v^*).$$ Since $S$ is not an independent set of $G_1'$ and since $G_1'$ is critical, $G_1' \div S$ is connected. Moreover, since $G_2'$ is critical, $G_2' \div v^*$ is connected. This proves that $G \div S$ has exactly two components $H_1$ and $H_2$ as claimed and the proof is complete.
\end{proof}

\begin{theorem}\label{theorem_haj-sum}
Let $G=(G_1,v_1,e_1) \Delta (G_2,v_2,e_2)$ be a Haj\'os join of two hypergraphs $G_1$ and $G_2$, and let $k \geq 2$ be an integer. Then, the following statements hold:
\begin{itemize}
\item[\upshape (a)] If both $G_1$ and $G_2$ are $(k+1)$-critical, then $G$ is $(k+1)$-critical.
\item[\upshape (b)] If $G$ is $(k+1)$-critical and $k \geq 3$, then both $G_1$ and $G_2$ are $(k+1)$-critical.
\end{itemize}
\end{theorem}

\begin{proof}
For the proof of (a), assume that both $G_1$ and $G_2$ are $(k+1)$-critical. If there is a $k$-coloring $\varphi$ of $G$, then there are vertices $x \neq y$ from $e^*$ such that $\varphi(x) \neq \varphi(y)$ and at least one vertex, say $x$, satisfies $\varphi(x) \neq \varphi (v^*)$. By symmetry, we may assume $x \in V(G_1)$. However, then the mapping $\varphi_1$ with $\varphi_1(u)=\varphi(u)$ for all $u \in V(G_1) \setminus \{v_1\}$ and $\varphi_1(v_1)=\varphi(v^*)$ is a $k$-coloring of $G_1$ and, thus, $\chi(G_1) \leq k$, a contradiction. In order to see that $G$ is $k$-critical, let $G'=G-e$ for some edge $e \in E(G)$. If $e=e^*$, then, as $G_1$ and $G_2$ are critical, we can create a $k$-coloring $\varphi$ of $G'$ by choosing $k$-colorings $\varphi_1$ of $G_1-e_1$ and $\varphi_2$ of $G_2-e_2$, permuting the colors such that $\varphi_1(v_1)=\varphi_2(v_2)$, and setting $\varphi(u)=\varphi_i(u)$ if $u \in V(G_i)$. If $e \neq e^*$, then $e \in E(G_i)$ for some $i \in \{1,2\}$, say $e \in E(G_1)$. Then, $G_1-e$ admits a $k$-coloring $\varphi_1$ and there is a vertex $u \in e_1$ with $\varphi_1(u) \neq \varphi_1(v_1)$. Moreover, $G_2-e_2$ admits a $k$-coloring $\varphi_2$ and all vertices from $e_2$ have the same color. Again by permuting the colors it is easy to see that one can create a $k$-coloring of $G$. Thus $G$ is $(k+1)$-critical, and (a) is proved.

In order to prove (b) assume that $G$ is $(k+1)$-critical with $k \geq 3$. By symmetry, it suffices to show that $G_1$ is $(k+1)$-critical, as well. Clearly, if $\chi(G_1) \leq k$, then there is a $k$-coloring $\varphi_1$ of $G_1$ with $\varphi_1(u) = \alpha \neq \beta=\varphi_1(v_1)$ for at least one $u \in e_1$. Moreover, as $G$ is $(k+1)$-critical and since $k \geq 3$, there is a $k$-coloring of $G-\tilde{e}$ and hence a $k$-coloring
$\varphi_2$ of $G_2-e_2$ such that $\varphi_2(v_2)=\beta$ and $\varphi_2(u') \neq \alpha$ for at least one $u' \in e_2 \setminus \{v_2\}$. Then, the union of the colorings $\varphi_1$ and $\varphi_2$ would be a $k$-coloring of $G$, a contradiction. Thus, $\chi(G_1) \geq k+1$. Similarly, one can show that $\chi(G_2) \geq k+1$. Now let $G_1'=G_1-e$ for some $e \in E(G_1)$. If $e=e_1$, then the restriction of any $k$-coloring $\varphi$ of $G-e^*$ to $V(G_1)$ is a $k$-coloring of $G_1'$ and we are done. If $e \neq e_1$, then there is a $k$-coloring $\varphi$ of $G-e$. If $\varphi(u) \neq \varphi(v^*)$ for at least one $u \in e^* \cap V(G_1)$, we are done. Otherwise, there is a vertex $u \in e \cap V(G_2)$ with $\varphi(u) \neq \varphi(v^*)$ and the restriction of $\varphi$ to $V(G_2)$ is a $k$-coloring of $G_2$, a contradiction to $\chi(G_2)\geq k+1$. This proves (b).
\end{proof}

Note that (b) does not hold for $k=2$, not even in the graph case as demonstrated for example by a cycle $C_7$ being obtained as Haj\'os join of two cycles $C_4$.

Let $G$ be a connected hypergraph, $v \in V(G)$, and $e \in E(G)$. Then, $\{v,e\}$ is a \textbf{separating set} (consisting of one edge and one vertex) if $v$ is a separating vertex of $G-e$ (no matter whether $v \in e$ or not).

\begin{theorem}\label{theorem_sep-set}
Let $G$ be a $(k+1)$-critical hypergraph with $k \geq 3$. If $G$ has a separating set consisting of one edge and one vertex, then $G$ is a Haj\'os join of two hypergraphs.
\end{theorem}

\begin{proof}
There is a vertex $v^* \in V(G)$ and an edge $e^* \in E(G)$ such that $G-e^*=G_1 \cup G_2$ with $V(G_1) \cap V(G_2) = \{v^*\}$ and $|G_i| \geq 2$ for $i \in \{1,2\}$. As $G$ is a block (by Proposition~\ref{prop_basic-facts}(c)), $e^* \cap V(G_i) \neq \varnothing$ for $i \in \{1,2\}$. For $i \in \{1,2\}$, let $e_i = (e^* \cap V(G_i)) \cup \{v^*\}$. If we can show that $e_i \not \in E(G)$, then $G$ is the Haj\'os join of $G_1 + e_1$ and $G_2 + e_2$, and we are done. By symmetry, assume that $e_1 \in E(G)$. As $G$ is $(k+1)$-critical, there is a $k$-coloring $\varphi$ of $G-e^*$ and all vertices from $e^*$ have the same color $\alpha$. Moreover, as $e_1 \in G$, $v^*$ has a color $\beta \neq \alpha$. Since $k \geq 3$, there is a color $\gamma \not \in \{\alpha, \beta\}$. By coloring all vertices from $G_2$ having color $\alpha$ with $\gamma$ and vice versa, we obtain a $k$-coloring of $G$, a contradiction. This completes the proof.
\end{proof}

The next theorem examines decompositions of $(k+1)$-critical hypergraphs having a separating edge set of size $k$.
Let $G$ be an arbitrary hypergraph. An \textbf{edge cut} of $G$ is a triple $(X,Y,F)$ such that $X$ is a non-empty proper subset of $V(G)$, $Y=V(G) \setminus X$, and $F=\partial_G(X)=\partial_G(Y)$. If $(X,Y,F)$ is an edge cut of $G$, by $X_F$ (respectively $Y_F$) we denote the set of vertices of $X$ that are incident to some edge of $F$. An edge cut $(X,Y,F)$ of $G$ is non-trivial if $|X_F| \geq 2$ and $|Y_F| \geq 2$.

That a $(k+1)$-critical graph is $k$-edge-connected was proved by Dirac~\cite{Dir53}. A characterization of $(k+1)$-critical graphs having a separating edge set of size $k$ was given by Toft \cite{Toft70} and, independently, by Gallai (oral communication to the third author). Gallai used the following lemma about complements of bipartite graphs. The \textbf{clique number} $\om(G)$ of a graph $G$ is the maximum integer $n$ such that $K_n$ is a subgraph of $G$. A graph $G$ is \textbf{perfect} if each induced subgraph $H$ of $G$ satisfies $\chi(H)=\om(H)$. It is well known that complements of bipartite graphs are perfect. For the reader's convenience we repeat the proof of the following lemma from \cite{StiToft18}.

\begin{lemma}
Let $H$ be a graph and let $k\geq 3$ be an integer. Suppose that $(A,B,F')$ is an edge cut of $H$ such that $|F'|\leq k$ and $A$ as well as $B$ are cliques of $H$ with $|A|=|B|=k$. If $\chi(H)\geq k+1$, then $|F'|=k$ and $F'=\partial_H(v)$ for some vertex $v$ of $H$.
\label{Le:perfect}
\end{lemma}
\begin{proof}
The graph $H$ is perfect and so $\om(H)=\chi(H)\geq k+1$. Consequently, $H$ contains a clique $X$ with $|X|=k+1$. Let $s=|A\cap X|$ and hence $k+1-s=|B\cap X|$. Since $|A|=|B|=k$, this implies that $s\geq 1$ and $k+1-s\geq 1$. Since $X$ is a clique of $H$, the set $E'$ of edges of $H$ joining a vertex of $A\cap X$ with a vertex of $B\cap X$ satisfies $E'\subseteq F'$ and $|E'|=s(k+1-s)$. Clearly, the function $g(s)=s(k+1-s)$ is strictly concave on the real interval $[1,k]$ as $g''(s)=-2$. Since $g(1)=g(k)=k$, we conclude that $g(s)>k$ for all $s\in (1,k)$. Since $g(s)=|E'|\leq |F'|\leq k$, this implies that $s=1$ or $s=k$. In both cases we obtain that $E'=F'=\partial_H(v)$ for some vertex $v$ of $H$ and $|E'|=|F'|=k$.
\end{proof}

\begin{theorem} \label{theorem_edge-cut}
 Let $G$ be a $(k+1)$-critical hypergraph with $k \geq 2$, and let $F \subseteq E(G)$ be a separating edge set of $G$ with $|F| \leq k$. Then, $|F|=k$ and there is an edge cut $(X,Y,F)$ of $G$ satisfying the following properties:
\begin{itemize}
\item[\upshape (a)] Every $k$-coloring $\varphi$ of $G[X]$ satisfies $|\varphi(X_F)|=1$ and every $k$-coloring $\varphi$ of $G[Y]$ satisfies $|\varphi(Y_F)|=k$ and for every color $i \in \{1,2,\ldots,k\}$ there is an edge $e \in F$ such that $\varphi(e \cap Y)=\{i\}$.
\item[\upshape (b)] Each vertex of $Y_F$ is incident to exactly one edge of $F$.

\item[\upshape (c)] If $|X_F| \geq 2$, then the hypergraph $G_1$ obtained from $G[X]$ by adding the hyperedge with vertex set $X_F$ is $(k+1)$-critical.
\item[\upshape (d)] The hypergraph $G_2$ obtained from $G[Y]$ by adding a new vertex $v$ and adding for each edge $e \in F$ the new edge $(e-X) \cup \{v\}$ is $(k+1)$-critical.
\end{itemize}
\end{theorem}
\begin{proof}
We may assume that $F$ is a minimal separating edge set of $G$ and, hence, there exists an edge cut $(X,Y,F)$ of $G$. Since $G$ is $(k+1)$-critical, for every set $Z\in \{X,Y\}$ there is a coloring $\varphi_Z\in \mathcal{CO}_k(G[Z])$. Now we construct an auxiliary graph $H$ as follows. The vertex set of $H$ consists of two disjoint cliques $A$ and $B$ with $|A|=|B|=k$, say $A=\{a_1,a_2, \ldots, a_k\}$ and $B=\{b_1,b_2, \ldots, b_k\}$. The edge set of $H$ consists of the edges of the cliques $A$ and $B$ and an additional edge set $F'\subseteq \partial_H(A)=\partial_H(B)$. An edge $a_ib_j$ belongs to $F'$ if and only if there is an edge $e \in F$ such that $\varphi_X(e \cap X) = \{i\}$ and $\varphi_Y(e \cap Y)=\{j\}$. We claim that $\chi(H)\geq k+1$. For otherwise, there exists a coloring $\varphi'\in \mathcal{CO}_k(H)$ and we may assume that $\varphi'(a_i)=i$ and $\varphi'(b_j)=\pi(j)$ for a permutation $\pi\in S_k$. Then  $\varphi_Y'=\pi \circ \varphi_Y$ belongs to $\mathcal{CO}_k(G[Y])$ and the function $\varphi_X \cup \varphi_Y'$ belongs to $\mathcal{CO}_k(G)$, which is impossible. This proves the claim that $\chi(H)\geq k+1$. From Lemma~\ref{Le:perfect} it then follows that $|F'|=k$ and $F'=\partial_H(v)$ for some vertex $v\in V(H)=A\cup B$. By symmetry, we may assume that $v\in A$. On the one hand, this implies that $X_F$ is an independent set of $G$ and $|\varphi_X(X_F)|=1$. On the other hand, it implies that $|\varphi_Y(Y_F)|=k$ and for every color $i \in \{1,2,\ldots,k\}$ there is an edge $e \in F$ such that $\varphi_Y(e \cap Y)=\{i\}$. This shows, in particular, that $|F|=k$. If $\varphi\in \mathcal{CO}_k(G[X])$ we can apply the same argument to the colorings $\varphi$ and $\varphi_Y$, which leads to $|\varphi(X_F)|=1$. If $\varphi \in \mathcal{CO}_k(G[Y])$, we can apply the same argument to the colorings $\varphi_X$ and $\varphi$, which leads to $|\varphi(Y_F)|=k$. This proves (a) and (b).

For the proof of (c) assume that $|X_F| \geq 2$ and let $G_1$ be the hypergraph obtained from $G[X]$ by adding the hyperedge with vertex set $X_F$. By (a), $\chi(G_1) \geq k+1$. Let $e$ be an arbitrary edge from $G_1$. We show that $G-e$ has a $k$-coloring. If $e=X_F$, this is evident. Otherwise, $e$ belongs to $G[X]$ and since $G$ is $(k+1)$-critical, there is a $k$-coloring $\varphi$ of $G-e$. Clearly, $\varphi$ induces a $k$-coloring of $G[Y]$ and we conclude from (a) that $|\varphi(X_F)|\geq 2$. Hence, $\varphi$ induces a $k$-coloring of $G_1-e$. Consequently, $G_1$ is $(k+1)$-critical (see Proposition~\ref{prop_G-e}).

In order to prove statement (d) let $G_2$ be the hypergraph obtained from $G[Y]$ by adding a new vertex $v$ and adding for each edge $e \in F$ the new edge $(e-X) \cup \{v\}$. By (a), $\chi(G_2) \geq k+1$. Let $e$ be an arbitrary edge of $G_2$. We show that $G_2 - e$ admits a $k$-coloring. Let $e'$ be the corresponding edge of $e$ in $G$. Then, $e' \in F \cup E(G[Y])$. As $G$ is $(k+1)$-critical, there is a $k$-coloring $\varphi$ of $G-e'$ and, by (a), $|\varphi(X_F)|=1$. Hence, $\varphi$ induces a $k$-coloring of $G_2-e$ and we are done.
\end{proof}

\section{Proof of Theorem~\ref{theorem_main-result}}
Let $G$ be hypergraph with $\lambda(G)\geq 3$. Then, $G$ contains a critical hypergraph $H$ with $\chi(G)=\chi(H)$. Furthermore, $\chi(H) \leq \lambda(H) + 1$ (by Proposition~\ref{prop_basic-facts}(b), respectively by Theorem~\ref{theorem_edge-cut} and Theorem~\ref{theorem_menger}). As $\lambda$ is a monotone hypergraph parameter, i.e., $\lambda(H) \leq \lambda(G)$ for any subhypergraph $H \subseteq G$, it follows $\chi(G) \leq \lambda(G) + 1$ and the first part of the main result is proven.

It remains to be shown that $\chi(G) = \lambda(G) + 1$ if and only if some block of $G$ belongs to $\mathcal{H}_{\lambda(G)}$. We will show that the critical subhypergraph $H$ is a block of $G$ which belongs to $\mathcal{H}_{\lambda(G)}$. For an integer $k \geq 2$, let $\mathcal{C}_k$ denote the class of hypergraphs $H$ such that $H$ is a critical hypergraph with chromatic  number $k+1$ and with $\lambda(H) \leq k$. We first prove that $\mathcal{C}_k=\mathcal{H}_k$.

\begin{theorem} \label{theorem_ck=hk}
Let $k\geq 3$ be an integer. Then, the two classes $\mathcal{C}_k$ and $\mathcal{H}_k$ coincide.
\end{theorem}

\begin{proof}
The proof of Theorem~\ref{theorem_ck=hk} is divided into five claims. Proving the following claim is straightforward and therefore left to the reader.
\begin{claim} \label{claim_odd-wheel}
The odd wheels belong to the class $\mathcal{C}_3$ and the complete graphs of order $k+1$ belong to the class $\mathcal{C}_k$.
\end{claim}

\begin{claim} \label{claim_ck-hajos}
Let $k \geq 3$ be an integer, and let $G=G_1 \Delta G_2$ be the Haj\'os join of two hypergraphs $G_1$ and $G_2$. Then, $G$ belongs to the class $\mathcal{C}_k$ if and only if both $G_1$ and $G_2$ belong to the class $\mathcal{C}_k$.
\end{claim}
\begin{proof2}
We may assume that $G=(G_1,v_1,e_1) \Delta (G_2,v_2,e_2)$. First suppose that $G_1$ and $G_2$ are from $\mathcal{C}_k$. Then, by Theorem~\ref{theorem_haj-sum}, $G$ is $(k+1)$-critical. It remains to be shown that $\lambda(G) \leq k$. To this end, let $u$ and $u'$ be distinct vertices of $G$ and let $p = \lambda_G(u,u')$. Then, there is a system $\mathcal{P}$ of $p$ edge disjoint $(u,u')$-hyperpaths in $G$. If $u$ and $u'$ are both from $G_1$, then only one hyperpath $P$ of $\mathcal{P}$ may contain vertices from $G_2$ (distinct from $v^*$). In this case, $P$ contains the vertex $v^*$ as well as the edge $e^*$. Let $u^* \in V(G_1)$ be the vertex from $P$ such that $u^*$ and $e^*$ are consecutive in $P$. Then, replacing the subhyperpath $u^*Pv^*$ of $P$ by the hyperpath $P'=(u^*,e_1,v_1)$ leads to a system of $p$ edge disjoint $(u,u')$-paths in $G_1$, and, thus, $p \leq \lambda_{G_1}(u,u') \leq k$. The same argument can be used if $u,u' \in V(G_2)$. It remains to consider the case that one vertex, say $u$, belongs to $G_1$ and the other vertex $u'$ belongs to $G_2$. By symmetry we may assume that $u \neq v^*$. Again at most one hyperpath $P$ of $\mathcal{P}$ uses the edge $e^*$ and all other hyperpaths of $\mathcal{P}$ contain the vertex $v^*(=v_1=v_2)$. As before, let $u^*$ be the vertex from $V(G_1)$ such that $u^*$ and $e^*$ are consecutive in $P$ and let $P'=(u^*,e_1,v_1)$. If we replace $P$ by the hyperpath $uPu^* + P'$, then we obtain $p$ edge disjoint $(u,v_1)$-hyperpaths in $G_1$, and thus, $p \leq \lambda_{G_1}(u,v_1) \leq k$. Hence, $\lambda(G) \leq k$ and so $G \in \mathcal{C}_k$.

Now suppose that $G \in \mathcal{C}_k$. As $k\geq 3$, it follows from Theorem~\ref{theorem_haj-sum}(b) that both $G_1$ and $G_2$ are $(k+1)$-critical graphs. It remains to be shown that $\lambda(G_i) \leq k$ for $i\in \{1,2\}$. By symmetry, it is sufficient to prove that $\lambda(G_1) \leq k$. Let $u$ and $u'$ be distinct vertices of $G_1$ and let $p=\lambda_{G_1}(u,u')$. Then, there is a system $\mathcal{P}$ of $p$ edge disjoint $(u,u')$-hyperpaths in $G_1$. At most one hyperpath $P$ of $\mathcal{P}$ may contain the edge $e_1$. If $v_1$ and $e_1$ are not consecutive in $P$, replacing $e_1$ by $e^*$ leads to a system of $p$ edge-disjoint $(u,u')$-hyperpaths of $G$ and so $p \leq \lambda_G(u,u') \leq k$ and we are done. So assume that $v_1$ and $e_1$ are consecutive in $P$. Let $u''$ be a vertex from $e_2 \setminus \{v_2\}$. As $G_2$ is critical, Proposition~\ref{prop_basic-facts}(b) implies that there is a $(u'',v_2)$-hyperpath $P'$, which does not contain the edge $e_2$. So, replacing the edge $e_1$ in $P$ by the sequence $e^*P'$, we get $p$ edge-disjoint $(u,u')$-hyperpaths of $G$, and hence, $p \leq \lambda_G(u,u') \leq k$. Thus, $\lambda(G_1) \leq k$ and the claim is proven.
\end{proof2}

The next claim is a direct consequence of claims \ref{claim_odd-wheel} and \ref{claim_ck-hajos}.

\begin{claim}
Let $k \geq 3$ be an integer. Then, $\mathcal{H}_k$ is a subclass of $\mathcal{C}_k$.
\end{claim}

\begin{claim}\label{claim_sep-size>2}
Let $k \geq 3$ be an integer, and let $G$ be a hypergraph from $\mathcal{C}_k$. If $G$ does not admit a separating vertex set of size at most $2$, then either $k=3$ and $G$ is an odd wheel, or $k \geq 4$ and $G$ is a complete graph of order $k+1$.
\end{claim}
\begin{proof2}
The proof is by contradiction; we consider a counter-example $G$ with minimum order $|G|$. Then, $G \in \mathcal{C}_k$ having no separating set of size at most $2$ and either $k=3$ and $G$ is not an odd wheel, or $k \geq 4$ and $G$ is not a complete graph of order $k+1$. First we show that the set $H$ of high vertices of $G$ contains at least two vertices. If $H=\varnothing$, then, as $G$ is a block and as $k \geq 3$, it follows from Lemma~\ref{lemma_G(X)-Gallai}(d) that $G$ is a complete graph of order $k+1$, a contradiction. If $|H|=1$, then Lemma~\ref{lemma_|G_H|=1} implies that $k=3$ and $G$ is an odd wheel, a contradiction. Thus, $|H| \geq 2$. Let $u$ and $v$ be distinct high vertices of $G$. As $G \in \mathcal{C}_k$, it follows from Proposition~\ref{prop_basic-facts}(b) that $\lambda(G)=k$ and, therefore, $G$ contains a separating edge set $F$ with $|F|=k$, which separates $u$ and $v$. From Theorem~\ref{theorem_edge-cut} it follows that there is an edge cut $(X,Y,F)$ satisfying the four properties of that theorem. Since $F$ separates $u$ and $v$, we may assume that $u \in X$ and $v \in Y$. As $u$ is a high vertex and $G$ has no separating vertex set of size at most two, it follows that $|X_F| \geq 3$. Now we consider the hypergraph $G_1$ obtained from $G[X]$ by adding the hyperedge $e$ with vertex set $X_F$. By Theorem~\ref{theorem_edge-cut}(c), $G_1$ is $(k+1)$-critical. As $G$ has no separating vertex set of size at most $2$ and since $|X_F| \geq 3$, $G_1$ has not neither.

Now we claim that $\lambda(G_1) \leq k$. To this end, let $x$ and $y$ be distinct vertices of $G_1$ and let $\mathcal{P}$ be a set of $p=\lambda_{G_1}(x,y)$ edge disjoint $(x,y)$-hyperpaths of $G_1$. Then, at most one hyperpath $P$ contains the edge $e$. The hyperpath $P$ contains a subhyperpath $P'=(z,e,z')$. Then, there is a $(z,z')$-hyperpath $P^*$ containing only edges of $F$ and $G[Y]$. This follows from Theorem~\ref{theorem_edge-cut}(d). By replacing the hyperpath $P'$ by $P^*$ we obtain a system of $p$ edge-disjoint $(x,y)$-hyperpaths in $G$ and so $p \leq \lambda_G(x,y) \leq k$. Hence, $\lambda(G_1) \leq k$ and so $G_1 \in \mathcal{C}_k$. Clearly, $|G_1| < |G|$ and either $k=3$ and $G_1$ is not an odd wheel, or $k \geq 4$ and $G_1$ is not a complete graph of order $k+1$. This gives a contradiction to the choice of $G$. Thus, the claim is proven.
\end{proof2}

\begin{claim}\label{claim_sep-size2}
Let $k \geq 3$ be an integer, and let $G$ be a hypergraph from $\mathcal{C}_k$. If $G$ has a separating vertex set of size $2$, then $G=G_1 \Delta G_2$ is the Haj\'os join of two hypergraphs $G_1$ and $G_2$, which both belong to $\mathcal{C}_k$.
\end{claim}
\begin{proof2}
If $G$ has a separating set consisting of one edge and one vertex, then Theorem~\ref{theorem_sep-set} implies that $G$ is the Haj\'os join of two hypergraphs $G_1$ and $G_2$. By Claim~\ref{claim_ck-hajos} it then follows that both $G_1$ and $G_2$ belong to $\mathcal{C}_k$ and we are done. It remains to consider the case that $G$ does not contain a separating set consisting of one edge and one vertex. By assumption, there is a separating vertex set of size $2$, say $S=\{v,w\}$. Then, Theorem~\ref{theorem_gal+dir} implies that $G \div S$ has exactly two components $H_1$ and $H_2$ such that the hypergraphs $G_i=G[V(H_i) \cup S]$ with $i \in \{1,2\}$ satisfy the three properties of this theorem. In particular, we get that $G_1'=G_1+vw$ is a $(k+1)$-critical hypergraph. By Proposition~\ref{prop_basic-facts}(b) it then follows that $\lambda_G(v,w) \geq k$ implying that $\lambda_{G_1}(v,w) \geq k-1$. As $G \in \mathcal{C}_k$, $\lambda_G(v,w) \leq k$, which implies that $\lambda_{G_2}(v,w) \leq 1$. Since $G_2$ is connected, this implies that $G_2$ has a separating edge $e$. But then, $\{v,e\}$ or $\{w,e\}$ is a separating set consisting of one edge and one vertex, a contradiction.
\end{proof2}

As a consequence of Claim~\ref{claim_sep-size>2} and Claim~\ref{claim_sep-size2}, we conclude that the class $\mathcal{C}_k$ is contained in the class $\mathcal{H}_k$ and so $\mathcal{C}_k=\mathcal{H}_k$, as claimed.
\end{proof}

\begin{proofof}[Theorem~\ref{theorem_main-result}]

In order to complete the proof of Theorem~\ref{theorem_main-result}, let $G$ be a hypergraph with $\lambda(G)=k$ and $k\geq 3$. As shown at the beginning of the section, we have $\chi(G) \leq k+1$. If one block $H$ of $G$ belongs to $\mathcal{H}_k$, then $H \in \mathcal{C}_k$ (by Theorem~\ref{theorem_ck=hk}) and hence $\chi(G) = k+1$ (by \eqref{eq_block-chi}).

Assume conversely that $\chi(G) = k+1$. Then, $G$ contains a critical subhypergraph $H$ such that $\chi(H) = k+1$. Since $\lambda(H) \leq \lambda(G) \leq k$, $H \in \mathcal{C}_k$. By Proposition~\ref{prop_basic-facts}(c), $H$ contains no separating vertex. We claim that $H$ is a block of $G$. Otherwise, $H$ would be a proper subhypergraph of a block $G'$ of $G$. This implies that there are distinct vertices $v$ and $w$ in $H$ which are joined by a hyperpath $P$ of $G$ satisfying $E(P) \cap E(H) = \varnothing$. Since $\lambda_H(v,w) \geq k$ (by Proposition~\ref{prop_basic-facts}(c)), this implies that $\lambda_G(v,w) \geq k+1$ and thus $\lambda(G) \geq k+1$, a contradiction. This proves the claim that $H$ is a block of $G$. As $\mathcal{C}_k=\mathcal{H}_k$ by Theorem~\ref{theorem_ck=hk}, it follows that $H \in \mathcal{H}_k$. This completes the proof of the theorem.

\end{proofof}

\section{Splitting Operation}

First we want to characterize the $(k+1)$-critical hypergraphs having a separating edge set of size $k$. These hypergraphs can be decomposed into smaller critical hypergraphs. We now want to introduce a reverse operation, called \textbf{splitting}.

Let $G_1$ and $G_2$ be two disjoint hypergraphs, let $\tilde{e} \in E(G_1)$ and $\tilde{v} \in V(G_2)$. Furthermore, let $s: \partial_{G_2}(\tilde{v}) \to 2^{\tilde{e}}$ be a mapping such that $s(e) \neq \varnothing$ for all $e \in \partial_{G_2}(\tilde{v})$ and $$\bigcup_{e \in \partial_{G_2}(\tilde{v})}s(e) = \tilde{e}.$$
Now let $G$ be the hypergraph with vertex set $V(G) = V(G_1) \cup (V(G_2) \setminus \{\tilde{v}\})$ and edge set $$E(G)=(E(G_1) \setminus \{\tilde{e}\}) \cup (E(G_2) \setminus \partial_{G_2}(v)) \cup \{(e- \{\tilde{v}\}) \cup s(e) ~|~ e \in \partial_{G_2}(\tilde{v})\}.$$
We then say that $G$ is obtained from $G_1$ and $G_2$ by \textbf{splitting} the vertex $\tilde{v}$ into the edge $\tilde{e}$, and we briefly write $G=S(G_1,\tilde{e},G_2,\tilde{v},s)$. If $|s(e)|=1$ for all $e \in \partial_{G_2}(\tilde{v})$, we call the splitting $s$ a \textbf{simple splitting}.

\begin{theorem}\label{theorem_splitting_low}
Let $G_1$ and $G_2$ be two disjoint $(k+1)$-critical hypergraphs with $k \geq 2$, let $\tilde{e} \in E(G_1)$, and let $\tilde{v} \in V(G_2)$ be a low vertex of $G_2$. Then the hypergraph $G=S(G_1,\tilde{e_1},G_2,\tilde{v},s)$ is $(k+1)$-critical, too, and $F=\partial_G(V(G_1))$ is a separating edge set of size $k$.
\end{theorem}

\begin{proof}
Since $\tilde{v}$ is a low vertex of $G_2$, for each coloring $\varphi \in \mathcal{CO}_k(G_2-\tilde{v})$ and for each color $i \in \{1,2,\ldots,k\}$ there is an edge $e \in \partial_{G_2}(\tilde{v})$ with $\varphi(e \setminus \{\tilde{v}\})=\{i\}$ (by Theorem~\ref{theorem_edge-cut}). Furthermore, in each coloring $\varphi$ of $G_1 - \tilde{e}$, the edge $\tilde{e}$ is monochromatic with respect to $\varphi$. Consequently, $\chi(G) \geq k+1$. It remains to show that $\chi(G-e) \leq k$ for all edges $e \in E(G)$. If $e \in E(G_1)$, then $G_1 - e$ admits a $k$-coloring $\varphi_1$ in which the edge $\tilde{e}$ is not monochromatic. Hence, we can choose any $k$-coloring $\varphi_2$ of $G_2 - \tilde{v}$ and permute the colors such that $\varphi_1 \cup \varphi_2$ is a $k$-coloring of $G-e$ (see Lemma~\ref{Le:perfect}). If $e \not \in E(G_1)$, we choose the corresponding edge $e' \in E(G_2)$. Then, there is a coloring $\varphi_2 \in \mathcal{CO}_k(G_2-e')$. Combining $\varphi_2$ with a coloring $\varphi_1 \in \mathcal{CO}_k(G_1 - \tilde{e})$ results in a $k$-coloring of $G-e$. Thus, $G$ is $(k+1)$-critical (see Proposition~\ref{prop_G-e}). By construction, $F$ is a separating edge set with $|F|=d_{G_2}(\tilde{v})=k$. This completes the proof.
\end{proof}

Combining Theorem~\ref{theorem_gal+dir} with the next results provides a characterization of $(k+1)$-critical hypergraphs having a separating vertex set of size $2$.

\begin{theorem}\label{theorem_splitting_ordinary}
Let $G_1$ and $G_2$ be two disjoint $(k+1)$-critical hypergraphs with $k \geq 2$, let $\tilde{e} \in E(G_1)$ be an ordinary edge of $G_1$, and let $\tilde{v} \in V(G_2)$ be an arbitrary vertex. Let $G=S(G_1,\tilde{e},G_2,\tilde{v},s)$ and let $G_2'=G[(V(G_2)\sm \{\tilde{v}\}) \cup \tilde{e}]$. If $\cn(G_2')\leq k$, then $G$ is a $(k+1)$-critical hypergraph and $\tilde{e}$ is a separating vertex set of $G$ of size $2$.
\end{theorem}
\begin{proof}
Let $\tilde{e}=uw$ and $G_1'=G_1-\tilde{e}$. Then, $G$ is the union of the two induced subgraphs $G_1'$ and $G_2'$ with $V(G_1')\cap V(G_2')=\{u,w\}$ and $|G_i'|>2$ as $|G_i|\geq k+1\geq 3$. So $S=\{u,w\}$ is a separating set of $G$. Furthermore, $G_1$ is obtained from $G_1'$ by adding the edge $uw$, and $G_2'$ is obtained from $G_2$ by identifying $u$ and $v$ to the new vertex $\tilde{v}$. Since $\cn(G_2)=k+1$ and $\cn(G_2')\leq k$, each coloring $\f_2\in \cC_k(G_2')$ satisfies $\f_2(u)\not= \f_2(w)$. Since $G_1$ is $(k+1)$-critical and $G_1'=G_1-uw$, each coloring $\f_1\in \cC_k(G_1')$ satisfies $\f_1(u)=\f_1(w)$. Consequently, $\cn(G)\geq k+1$. Now let $e$ be an arbitrary edge of $G$. It remains to show that $\cn(G-e)\leq k$. First assume that $e$ belongs to $G_1'$ and hence to $G_1$. As $G_1$ is $(k+1)$-critical, there is a coloring $\f_1\in \cC_k(G_1-e)$ and so $\f_1(u)\not=\f_1(w)$. There is a coloring $\f_2\in \cC_k(G_2')$ and $\f_2(u)\not=\f_2(w)$. By permuting colors if necessary, $\f_1 \cup \f_2$ is a $k$-coloring of $G-e$. Now assume that $e$ belongs to $G_2'$ and let $e'$ be the corresponding edge of $G_2$. As $G_2$ is $(k+1)$-critical, there is a coloring $\f_2\in \cC_k(G_2-e')$ which leads to a coloring $\f_2'\in \cC_k(G_2'-e)$ such that $\f_2'(u)=\f_2'(w)=\f_2(\tilde{v})$.  As $G_1$ is $(k+1)$-critical, there is a coloring $\f_1\in \cC_k(G_1-\tilde{e})$ and so $\f_1(u)=\f_1(w)$.  By permuting colors if necessary, $\f_1 \cup \f_2'$ yields a $k$-coloring of $G-e$. Hence $G$ is $(k+1)$-critical (by Proposition~\ref{prop_G-e}).
\end{proof}

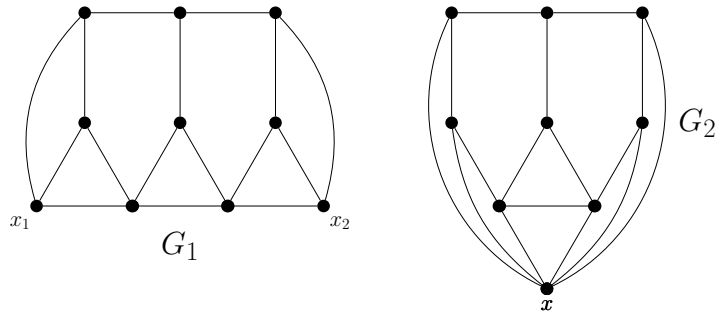
\begin{figure}[htbp]
\centering
\usetikzlibrary[graphs, arrows, backgrounds, intersections, positioning, fit, petri, calc, shapes, decorations.pathmorphing]


%
\resizebox{.7\linewidth}{!}{
\begin{tikzpicture}[node distance=1cm, bend angle=30,
vertex/.style={circle,minimum size=3mm,very thick, draw=black, fill=black, inner sep=0mm}, information text/.style={fill=red!10,inner sep=1ex, font=\Large}, help lines/.style={-,color=black}]

\node[draw=none,minimum size=3cm,regular polygon,regular polygon sides=3] (a) {};
\foreach \x in {1,2,3}
\node[vertex] (a\x) at (a.corner \x) {};

\node[draw=none,minimum size=3cm,regular polygon,regular polygon sides=3, xshift=2.6cm](b){};
\foreach \x in {1,2,3}
\node[vertex] (b\x) at (b.corner \x) {};

\node[draw=none,minimum size=3cm,regular polygon,regular polygon sides=3, xshift=5.2cm](c){};
\foreach \x in {1,2,3}
\node[vertex] (c\x) at (c.corner \x) {};

\node[vertex] (v1) at (a1) [yshift=3cm]{};
\node[vertex] (v2) at (b1) [yshift=3cm]{};
\node[vertex] (v3) at (c1) [yshift=3cm]{};

\path[-]
  (a1) edge [help lines] (a2)
  	   edge [help lines] (a3)
  	   edge [help lines] (v1)
  (b1) edge [help lines] (b2)
  	   edge [help lines] (b3)
  	   edge [help lines] (v2)
  (c1) edge [help lines] (c2)
  	   edge [help lines] (c3)
  	   edge [help lines] (v3)
  (a2) edge [help lines] (a3)
  (b2) edge [help lines] (b3)
  (c2) edge [help lines] (c3)	   
  (v1) edge [help lines] (v2)
  	   edge [help lines, bend right] (a2)
  (v3) edge [help lines] (v2)
  	   edge [help lines, bend left] (c3);     

\node at (a2) [label={[font=\Large]260:$x_1$}] {};
\node at (c3) [label={[font=\Large]-80:$x_2$}] {};
\node at (b1) [label={[font=\huge, yshift=-2.7cm]below:$G_1$}] {};

\begin{scope}[xshift=10cm]
\node[draw=none,minimum size=3cm,regular polygon,regular polygon sides=3] (a) {};
\foreach \x in {1,3}
\node[vertex] (a\x) at (a.corner \x) {};

\node[draw=none,minimum size=3cm,regular polygon,regular polygon sides=3, xshift=2.6cm](b){};
\foreach \x in {1,2,3}
\node[vertex] (b\x) at (b.corner \x) {};

\node[draw=none,minimum size=3cm,regular polygon,regular polygon sides=3, xshift=5.2cm](c){};
\foreach \x in {1,2}
\node[vertex] (c\x) at (c.corner \x) {};

\node[vertex] (v1) at (a1) [yshift=3cm]{};
\node[vertex] (v2) at (b1) [yshift=3cm]{};
\node[vertex] (v3) at (c1) [yshift=3cm]{};

\node[draw=none,minimum size=3cm,regular polygon,regular polygon sides=3, rotate=60, yshift=-3cm](w){};
\foreach \x in {1,2,3}
\node[vertex, label={[font=\Large]below:$x$}] (w1) at (w.corner 2) {};

\path[-]
  (a1) 
  	   edge [help lines] (a3)
  	   edge [help lines] (v1)
  	   edge [help lines, bend angle=22.5, bend right] (w1)
  (b1) edge [help lines] (b2)
  	   edge [help lines] (b3)
  	   edge [help lines] (v2)
  (c1) edge [help lines] (c2)
  	   edge [help lines] (v3)
  	   edge [help lines, bend angle=22.5, bend left] (w1)
  (b2) edge [help lines] (b3)
  	   edge [help lines] (w1)
  (b3) edge [help lines] (w1)	   
  (v1) edge [help lines] (v2)
  	   edge [help lines, bend angle=45, bend right] (w1)
  (v3) edge [help lines] (v2)
  	   edge [help lines, bend angle=45, bend left] (w1);     

\node at (c1) [xshift=1.5cm, font=\huge] {$G_2$};

\end{scope}

\end{tikzpicture}}
%
\caption{Two $4$-critical graphs.}
\label{fig:splitting}
\end{figure}

There are $(k+1)$-critical graphs $G_2$ and vertices $v$ of $G_2$ such that the resulting graph $G_2'$ obtained from $G_2$ by splitting $v$ into an independent set of size at least $2$ satisfies $\cn(G_2')\geq k+1$; in this case $G_2'$ is $(k+1)$-critcal, too. An example with $k=3$ is shown in Figure~\ref{fig:splitting}; both graphs $G_1$ and $G_2$ are $4$-critical and $G_1$ is obtained from $G_2$ by splitting $x$ into the vertex set $\{x_1,x_2\}$. The graph $G_1$ is a Haj\'os join of the form $G=(K_4\triangle K_4) \triangle K_4$ and hence $4$-critical. That $G_2$ is  $4$-critical can also easily be checked by hand using Proposition~\ref{prop_G-e}.

Both Theorems \ref{theorem_splitting_low} and \ref{theorem_splitting_ordinary} are special cases of a more general theorem about the splitting operation for critical hypergraphs. The proof of the next result is almost the same as the proof of the former theorem.

\begin{theorem}\label{theorem_splitting_general}
Let $G_1$ and $G_2$ be two disjoint $(k+1)$-critical hypergraphs with $k \geq 2$, let $\tilde{e} \in E(G_1)$ be an arbitrary edge of $G_1$, and let $\tilde{v} \in V(G_2)$ be an arbitrary vertex. Let $G=S(G_1,\tilde{e},G_2,\tilde{v},s)$ and let $G_2'=G[(V(G_2)\sm \{\tilde{v}\}) \cup \tilde{e}]$. Assume that for every coloring $\f\in \cC_k(G[\tilde{e}])$ with $|\f(\tilde{e})|\geq 2$ there is a coloring $\f'\in \cC_k(G_2')$ such that $\f'|_{\tilde{e}}=\f$. Then, $G$ is a $(k+1)$-critical hypergraph.
\end{theorem}

A slightly weaker version of the above theorem was already proved by Toft \cite{Toft74}; he only considered the case when $G_2$ is a critical graph and $s$ is a simple splitting. Then, the resulting critical hypergraph $G$ has one hyperedge less. By repeated application of the splitting operation one can finally obtain a critical graph.

Let $G_1, G_2, \tilde{e}, \tilde{v}, G$ and $G_2'$ as in Theorem~\ref{theorem_splitting_general}. As $G_1$ is critical, $G_1$ is a simple hypergraph (by Proposition~\ref{prop_basic-facts}(d)). Hence, $\tilde{e}$ is an independent set of $G$ as well as of $G_2'$ and $G[\tilde{e}]=G_2'[\tilde{e}]$. We then say that $G_2'$ is obtained from $G_2$ by splitting $\tilde{v}$ into the independent set $\tilde{e}$, and write $G_2'=S(G_2,\tilde{v},\tilde{e},s)$.

Let $G$ be a $(k+1)$-critical hypergraph with $k\geq 2$, and let $v$ be a vertex of $G$. We say that $v$ is a \textbf{universal vertex} of $G$, if for every hypergraph $G'=S(G,v,X,s)$, where $X$ is a set, and every coloring $\f'\in \cC_k(G'[X])$ with $|\f'(X)|\geq 2$ there is a coloring $\f\in \cC_k(G)$ with $\f|_X=\f'$.

Theorem~\ref{theorem_splitting_general} then implies that if $G_1$ and $G_2$ are disjoint $(k+1)$-critical hypergraphs, and $\tilde{v}$ is a universal vertex of $G_2$, then any hypergraph $G$ obtained from $G_1$ and $G_2$ by splitting $\tilde{v}$ into an edge $\tilde{e}$ of $G_2$ is a $(k+1)$-critical hypergraph, too. However, a good characterization of universal vertices in critical hypergraphs or graphs seems not available. From the proof of Theorem~\ref{theorem_splitting_low} it follows that any low vertex of a $(k+1)$-critical hypergraph with $k\geq 2$ is universal. Further cases were given by Toft in \cite{Toft70} and \cite{Toft74}.

Next to the Haj\'os construction there is another construction for critical hypergraphs, first used by Dirac for critical graphs (see Gallai \cite[(2.1)]{Gal63a}). Let $G_1$ and $G_2$ be two disjoint hypergraphs, and let $G$ be the hypergraph obtained from the union $G_1 \cup G_2$ by adding all ordinary edges between $G_1$ and $G_2$, that is, $V(G)=V(G_1) \cup V(G_2)$ and $E(G)=E(G_1) \cup E(G_2) \cup \set{uv}{u\in V(G_1), v\in V(G_2)}$. We call $G$ the \textbf{Dirac sum}, or the \textbf{join} of $G_1$ and $G_2$ and write $G=G_1\boxtimes G_2$. Then it is straightforward to show that $\cn(G)=\cn(G_1)+\cn(G_2)$, and, moreover, $G$ is critical if and only if both $G_1$ and $G_2$ are critical. For example, $KC_{n,p}=K_n \boxtimes C_{2p+1}$ is a $(n+3)$-critical graph and, as proved by Toft \cite{Toft74}, each high vertex of $KC_{n,p}$ is universal. These graphs enable us to construct from any $(k+1)$-critical hypergraph with $k\geq 3$ and copies of $KC_{k-2,p}$ a $(k+1)$-critical graph. Note that if $G=S(G_1,\tilde{e},G_2,\tilde{v},s)$ and $s$ is a simple splitting, then $d_{G_2}(\tilde{v})\geq |\tilde{e}|$. One popular example of a critical graph obtained from a critical hypergraph was presented by Toft \cite{Toft70}. For $i\in \{1,2\}$, let $G_i$ be a connected hypergraph with one edge $e_i$ of size $2p+1$, so $G_i$ is a $2$-critical hypergraph. Then the Dirac sum $G'=G_1 \boxtimes G_2$ is a $4$-critical hypergraph. If we now apply the splitting operation with two copies of the odd wheels $KC_{1,p}$ and the high vertex $v$, that is, we first construct $\tilde{G}=S(G',e_1,KC_{1,p},v,s)$ with a simple splitting $s$ and then $G=S(\tilde{G},e_2,KC_{1,p},v,s')$ with a simple splitting $s'$, then the resulting graph $G$ is a $4$-critical graph of order $n=8p+4$ and with $m=(2p+1)^2+8p+4=\frac{1}{16}n^2+n$ edges, i.e., $G$ has many edges. The constant $\frac{1}{16}$ has not been improved.

\section{Concluding remarks}

Surprisingly, we are not able to characterize the hypergraphs with $\lambda=2$ and $\cn=3$. If $\cH_2$ denotes the smallest class of hypergraphs that contains all hyperwheels and is closed under taking Haj\'os joins, then it is easy to show that $\cH_2$ is contained in the class $\cc_2$ of $3$-critical hypergraphs with $\lambda\leq 2$. As proved in Claim~\ref{claim_sep-size>2} if $G$ belongs to $\cc_k$ with $k\geq 3$ and $G$ has no separating vertex set of size at most $2$, then $G$ is a base graph of $\cH_k$, that is, either $k=3$ and $G$ is an odd wheel or $k\geq 4$ and $G$ is a $K_{k+1}$. However, there are hypergraphs in $\cc_2$ that do not have a separating vertex set of size at most $2$, but that are different from hyperwheels.  Examples of such $3$-critical hypergraphs can be obtained as follows. Let $T$ be an arbitrary rooted  tree such that the root has degree at least $2$ and the  distance between the leafs of $T$ and the root all have the same parity. If $G$ is the hypergraph obtained from $T$ by adding the hyperedge consisting of the leafs of $T$, then it is easy to check that $G\in \cc_2$. If the non-leaf vertices of $T$ have degree at least $3$, then $G$ has no separating vertex set of size at most $2$; one such hypergraph is shown in Figure~\ref{fig:superwheel}. On the other hand, $G$ belongs to $\cH_2$, and we do not know any hypergraph belonging to $\cc_2$, but not to $\cH_2$. If $G\in \cc_2$ then $G$ has a separating edge set of size $2$, and  according to Theorem~\ref{theorem_edge-cut} the hypergraph $G$ can be decomposed into two $3$-critical hypergraphs $G_1$ and $G_2$. It can easily be shown that $\lambda(G_i)\leq 2$ for $i\in \{1,2\}$ implying that both $G_1$ and $G_2$ belong to $\cc_2$. The problem is the converse splitting operation.

\begin{figure}[htbp]
\centering
\usetikzlibrary[graphs, arrows, backgrounds, intersections, positioning, fit, petri, calc, shapes, decorations.pathmorphing]


%
\resizebox{.35\linewidth}{!}{
\begin{tikzpicture}[node distance=1cm, bend angle=30,
vertex/.style={circle,minimum size=3mm,very thick, draw=black, fill=black, inner sep=0mm}, information text/.style={fill=red!10,inner sep=1ex, font=\Large}, help lines/.style={-,color=black}]

\node[vertex] (v1) {};

\node[draw=none,minimum size=3cm,regular polygon,regular polygon sides=3] (a) {};
\foreach \x in {1,2,3}
\node[vertex] (a\x) at (a.corner \x) {};

\node at (v1) [draw=none,minimum size=2cm,regular polygon,regular polygon sides=3, yshift=2.5cm, rotate=60] (b) {};
\foreach \x in {1,2,3}
\node[vertex] (b\x) at (b.corner \x) {};

\node at (v1) [draw=none,minimum size=2cm,regular polygon,regular polygon sides=3, yshift=-1.25cm, xshift=-2.2cm, rotate=60] (c) {};
\foreach \x in {1,2,3}
\node[vertex] (c\x) at (c.corner \x) {};

\node at (v1) [draw=none,minimum size=2cm,regular polygon,regular polygon sides=3, yshift=-1.25cm, xshift=2.2cm, rotate=60] (d) {};
\foreach \x in {1,2,3}
\node[vertex] (d\x) at (d.corner \x) {};

\path[-]

(v1) edge [help lines] (a1)
	 edge [help lines] (a2)
	 edge [help lines] (a3)
(b2) edge [help lines] (b1)
	 edge [help lines] (b3)
(c3) edge [help lines] (c1)
	 edge [help lines] (c2)
(d1) edge [help lines] (d2)
	 edge [help lines] (d3);

\begin{pgfonlayer}{background}

\node[circle, minimum size=7.5cm, draw=black, fill=black!50, fill opacity=.2] at (v1){};
\node[circle, minimum size=5.5cm, draw=black, fill=white] at (v1) {};

\end{pgfonlayer}


%
%
%
%
%

\end{tikzpicture}}
%
\caption{A member in $\cc_2$ without a separating vertex set of size $2$.}
\label{fig:superwheel}
\end{figure}
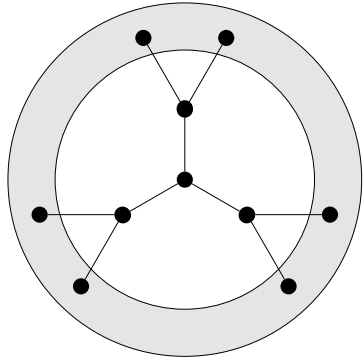

It seems likely that one can obtain a polynomial time algorithm from the proof of Theorem~\ref{theorem_main-result}, which, given a hypergraph $G$ with $\lambda(G)\leq k$ and $k\geq 3$, either finds a $k$-coloring of $G$ or a block belonging to $\cH_k$. We did not explore this question.

\end{document}